\documentclass{amsart}

\usepackage{amsfonts}
\usepackage{amsmath}
\usepackage{amssymb}
\usepackage{amscd}
\usepackage{mathrsfs}  
\usepackage{amsbsy}
\usepackage{amsthm}
\usepackage{graphicx}
\usepackage{fancyhdr}
\usepackage{color}
\usepackage{xy}
\usepackage{enumitem}
\usepackage{bbm}
\usepackage{upgreek,textgreek}
\usepackage{hyperref}
\usepackage{tikz-cd}

\usepackage[OT2,OT1]{fontenc}  
\newcommand\cyr{%
\renewcommand\rmdefault{wncyr}%
\renewcommand\cFdefault{wncyss}%
\renewcommand\encodingdefault{OT2}%
\normalfont \selectfont} \DeclareTextFontCommand{\textcyr}{\cyr}

\newcommand{\be}{\begin{equation}}
\newcommand{\ee}{\end{equation}}

\newcommand{\bes}{\begin{equation*}}
\newcommand{\ees}{\end{equation*}}

\newcommand{\C}{\mathbb{C}}

\newcommand{\K}{\mathbb{K}}
\newcommand{\N}{\mathbb{N}}

\newcommand{\R}{\mathbb{R}}

\newcommand{\cF}{\mathcal{F}}

\newcommand{\cW}{\mathcal{W}}

\newcommand{\sF}{\mathscr{F}}

\newcommand{\sfA}{\mathsf{A}}

\newcommand{\sfE}{\mathsf{E}}
\newcommand{\sfF}{\mathsf{F}}

\newcommand{\sfU}{\mathsf{U}}
\newcommand{\sfV}{\mathsf{V}}
\newcommand{\sfX}{\mathsf{X}}
\newcommand{\sfY}{\mathsf{Y}}

\newcommand{\fF}{\mathfrak{F}}

\newcommand{\id}{\mathrm{id\,}}

\newcommand{\nl}{\vskip 10pt\noindent}

\renewcommand{\rmdefault}{cmr} 

\newtheorem{theorem}{Theorem}[section]

\theoremstyle{definition}
\newtheorem{definition}[theorem]{Definition}
\newtheorem{example}[theorem]{Example}

\DeclareMathOperator\Lip{Lip}

\DeclareMathOperator\Int{int}

\DeclareMathOperator\Diff{Diff}
\DeclareMathOperator\Hom{Hom}

\newcommand{\sH}{\mathscr{H}}

\newcommand{\oL}{\overline{L}}

\newcommand{\n}[1]{{\left\|{#1}\right\|}}

\newcommand{\vertiii}[1]{{\left\vert\kern-0.25ex\left\vert\kern-0.25ex\left\vert #1 
    \right\vert\kern-0.25ex\right\vert\kern-0.25ex\right\vert}}

\newcommand{\one}{1\hspace{-0.23em}\mathrm{l}}

\theoremstyle{remark}
\newtheorem{remark}[theorem]{Remark}
\newtheorem{corollary}[theorem]{Corollary}

\numberwithin{equation}{section}

\newcommand{\abs}[1]{\lvert#1\rvert}

\begin{document}

\title{Fractal Interpolation over Curves}

\author{Peter R. Massopust}
\address{Centre of Mathematics, Technical University of Munich, Boltzmannstrasse 3, 85747 Garching b. M\"unchen, Germany}
\email{massopust@ma.tum.de}


\date{January 1, 1994 and, in revised form, June 22, 1994.}



\begin{abstract}
This paper introduces the novel concept of fractal interpolation over curves in Banach spaces. The contents are based on the usual methodologies involving the fractal interpolation problem over intervals but the current approach considerably extends and generalizes them and paves the way to fractal interpolation over Banach submanifolds.
\end{abstract}

\maketitle



\section{Introduction}
Fractal interpolation techniques and methodologies have been considered extensively over the last decades in order to obtain approximation results for immensely complex and complicated geometries such as fractal sets or sets modelled as fractal sets. Common to all these techniques is the existence or the approximate assumption of an underlying self-referential structure of the (approximate) fractal set. 

For many application, the (approximate) fractal set is the graph of a function defined on an appropriate subset of a, say, Banach space with values in another Banach space. In other words, one tries to solve the Fractal Interpolation Problem: Given a bounded subset $\sfX$ of a Banach space $\sfE$ and a Banach space $\sfF$, construct a \emph{global} function $\psi:\sfX = \coprod\limits_{i=1}^n \sfX_i\to\sfF$ belonging to some prescribed function space $\sF:= \sF(\sfX,\sfF)$ satisfying $n$ functional equations of the form
\[
\psi (h_i (x)) = q_i (x) + s_i (x) \psi (x), \quad\text{on $\sfX$ and for $i\in \N_n$}.
\]
The functions $h_i$ are assumed to partition $\sfX$ (linearly or nonlinearly) into disjoint subsets $\sfX_i = h_i(\sfX)$, and the functions $q_i$ and $s_i$ are appropriately chosen. We remark that, if such a global solution exists, then it is pieced together in a prescribed manner from copies of itself on the subsets $\sfX_i$. This latter property then defines or expresses the self-referential nature of $\psi$ and communicates the fact that the graph of $\psi$ is in general an object of immense geometric complexity, i.e., a fractal set.

The afore-mentioned fractal interpolation problem has been investigated in numerous publications and for different function spaces $\sF$ where the domains of $f\in \sF$ were one-dimensional, $n$-dimensional, or infinite-dimensional. For specific applications though, $\sfX$ is usually a compact or half-open bounded interval of the numerical Banach space $\R$ with values in $\sfF$ or even just $\R^m$, $m\in\N$. 

However, in many situations especially in physics and particularly electrodynamics, complex (approximately) self-referential structures may arise over domains which are curves or more generally (Banach) submanifolds. In this paper, we initiate the investigation in this latter setting and introduce fractal interpolation where the domain $\sfX$ is the trace of a regular curve $\gamma :I \to \sfE$ in a Banach space. (For a proper choice of the interval $I$ the trace is then a $C^1$ Banach submanifold of codimension one.) The set-up presented here will then indicate the general approach to fractal interpolation over (Banach) submanifolds and will be left to a follow-up paper.

The contents of this paper are organized as follows. In Section 2, iterated function systems are defined and the notation and terminology for the remainder of the paper introduced. In the next section, some remarks about the code space naturally associated with an IFS are presented. Section 4 reviews some concepts from the theory of curves in Banach spaces highlighting those results that are needed in this paper. In Section 5, the novel concept of fractal interpolation over curves in Banach spaces is introduced and an appropriate function space $\sF$, namely a certain Lipschitz algebra, chosen. The main result, namely the existence of a unique solution to the fractal interpolation problem over curves, is proven in this section as well. In the next section, a relation between the unique solution of the interpolation problem and the underlying IFS and its code space is exhibited. The final section lists some future research directions.

\section{Iterated Functions Systems}
%

Let $\sfE:=(\sfE, \n{\cdot}_\sfE)$ and $\sfF:=(\sfF, \n{\cdot}_\sfF)$ be Banach spaces (over $\R$ or $\C$). For a map $f: \sfE \to \sfF$, we define the Lipschitz constant associated with $f$ by
\[
L (f) = \sup_{x\in\sfE, x \neq 0} \frac{\n{f(x)-f(y)}_\sfF}{\n{x-y}_\sfE}.
\]
A map $f$ is called \emph{Lipschitz} or a \emph{Lipschitz function} if $L (f) < + \infty$ and a \emph{contraction} if $L (f) < 1$.

\begin{definition}
Let $(\sfE,\n{\cdot}_\sfE$ be a Banach space and $\cF$ a finite set of functions $\sfE\to\sfE$. Then, the pair $(\sfE,\cF)$ is called an iterated function system (IFS) on $\sfE$. In case all maps $f\in\cF$ are contractions then $(\sfE,\cF)$ is called a \emph{contractive} IFS.
\end{definition}

For a more general definition of IFS and related topics, we refer the interested reader to \cite{BWL14}.

\begin{remark}
Throughout this paper, we deal with contractive IFSs and therefore we drop the adjective ``contractive."
\end{remark}

With the finite set of contractions $\cF$ on $\sfE$, one associates a set-valued operator, again denoted by $\cF$, acting on the hyperspace ${\sH}(\sfE)$ of nonempty compact subsets of $\sfE$:
\[
\cF (E) := \bigcup_{f\in \cF} f (E),\qquad E\in \sH(\sfE).
\]
The hyperspace ${\sH}(\sfE)$ endowed with the Hausdorff-Pompeiu metric $d_{\sH}$ defined by
\[
d_{\sH} (S_1, S_2) := \max\{d(S_1, S_2), d(S_2, S_1)\},
\]
where $d(S_1,S_2) := \sup\limits_{x\in S_1} d(x, S_2) := \sup\limits_{x\in S_1}\inf\limits_{y\in S_2} \n{x-y}_\sfE$, becomes a metric space.

It is a known fact that the completeness of $\sfE$ implies the completeness of $({\sH}(\sfE), d_{\sH})$ as a metric space. Moreover, it can be shown that the set-valued operator $\cF$ is contractive on the complete metric space $({\sH}(\sfE), d_{\sH})$ with Lipschitz constant $L(\cF) = \max\{L (f) : f\in \cF\} < 1$ if all $f\in \cF$ are contractions. 

In this case and by the Banach Fixed Point Theorem, $\cF$ has a unique fixed point in $\sH(\sfE)$. This fixed point is called the \emph{attractor} of or the \emph{fractal (set)} generated by the IFS $(\sfE,\cF)$. The attractor or fractal $F$ satisfies the self-referential equation
\be\label{fixedpoint}
F = \cF (F) = \bigcup_{f\in\cF} f (F),
\ee
i.e., $F$ is made up of a finite number of images of itself. Eqn. \eqref{fixedpoint} reflects the fractal nature of $F$ showing that it is as an object of immense geometric complexity.

The proof of the Banach Fixed Point Theorem also shows that the fractal $F$ can be obtained iteratively via the following procedure: Choose an arbitrary $F_0\in {\sH}(\sfE)$ and define
\be\label{F}
F_n := \cF (F_{n-1}),\qquad n\in\N.
\ee
Then $F =\lim\limits_{n\to\infty} F_n$, where the limit is taken with respect to the Hausdorff-Pompeiu metric $d_{\sH}$.

For more details about IFSs and fractals and their properties, we refer the interested reader to the large literature on these topics and list only two references \cite{B12,M16} which are closely related to the present exhibition.
\section{Some Remarks about IFSs and Their Code Space}
In this short section, we give a brief overview of the important concept of code space associated with an IFS. As a reference, we offer \cite{B12,BBHV16}.

Let $(\sfE, \cF)$ be an IFS with $\#\cF = n$. With an IFS one associates the code space $\N_n^\infty$, i.e., the compact space of all infinite sequences of the form $\boldsymbol{i} := (i_1 i_2 i_3\ldots )$ whose elements are from $\N_n$. The set $\N_n^\infty$ is endowed with a metric $\delta$ defined by
\[
\delta (\boldsymbol{i},\boldsymbol{j}) := 2^{-k},
\]
where $k$ is the least integer such that $i_k\neq j_k$. This makes $(\N_n^\infty,\delta)$ into a compact metric space.

Define maps
\begin{gather*}
s_i:\N_n^\infty\to\N_n^\infty,\\
s_i(\boldsymbol{j}) := i\,\boldsymbol{j}. 
\end{gather*}
Then each $s_i$ is a contraction with Lipschitz constant $L(s_i) = \frac12$ and a homeomorphism onto its image. In other words, the pair $(\N_n^\infty, s)$ with $s := \{s_i : i\in \N_n\}$ is an IFS referred to as the code space IFS.

For an IFS $(\sfE, \cF)$ with maps $\cF = \{f_1, \ldots, f_n\}$ and attractor $F$, define the code map 
\begin{gather*}
\pi : \N_n^\infty\to F,\\
\pi(\boldsymbol{i}) := \lim_{k\to\infty} f_{i_1}\circ f_{i_2}\circ \cdots \circ f_{i_k} (x),
\end{gather*}
for a fixed $x\in F$ and all $\boldsymbol{i} = i_1 i_2 \ldots \in \N_n^\infty$. 

It is known that for contractive IFSs, $\pi(\boldsymbol{i})$ is one-elemental and independent of $x\in F$. Furthermore, $\pi$ is continuous and surjective. If
\begin{gather*}
S_i:\N_n^\infty\to\N_n^\infty,\\
S_i (i_1 i_2 i_3 \ldots) := (i_2 i_3 \ldots)
\end{gather*}
denotes the left shift operator on $\N_n^\infty$, then $\pi$ is a conjugation between the code space IFS $(\N_n^\infty, \sigma)$ and the IFS $(\sfE, \cF)$ in the sense that
\[
\pi\circ\sigma_i = f_i\circ\pi\quad\text{and}\quad \pi\circ S_i \in \cF^{-1}\circ \pi,
\]
where $\cF^{-1}(E) := \bigcup\limits_{i\in\N_n} f_i^{-1} (E)$, $E\in \sH(\sfE)$.

In order to introduce the concept of fractal transformation, we require one more item, namely, section of the code map $\pi$. A mapping $\sigma: F\to \N_n^\infty$ is called a section of $\pi$ if $\pi\circ\sigma = \id_F$.

Now suppose, two IFSs $(\sfE, \cF_1)$ and $(\sfE, \cF_2)$ with $\#\cF_1 = \#\cF_2$ are given. Denote by $F_1$ and $F_2$, respectively, their attractors. The fractal transformations $\fF_{12} : F_1 \to F_2$ and  $\fF_{21} : F_2 \to F_1$ are defined by
\[
\fF_{12} := \pi_2\circ\sigma_1\quad\text{and}\quad \fF_{21} := \pi_1\circ\sigma_2.
\]
Their respective actions are best seen by the commutative diagram below.
\[
\begin{tikzcd}
F_1 \arrow[dr, shift left=.5ex, "\sigma_1"] \arrow[rr, shift left=0.5ex,"\fF_{12}"] \arrow[rr, leftarrow, shift right=0.5ex, swap, "\fF_{21}"] && F_2 \arrow[dl, shift left=.5ex, "\sigma_2"]\\
& \N_n^\infty \arrow[ul, shift left=.5ex, "\pi_1"] \arrow[ur, shift left=.5ex, "\pi_2"]
\end{tikzcd}
\]
If $\fF_{12}$ is a homeomorphism then it is called a fractal homeomorphism and thus $\fF_{21} = (\fF_{12})^{-1}$.
\section{Curves in Banach Spaces}
In this section, we briefly introduce the concept of a curve in a normed linear space, in particular, a Banach space.
\begin{definition}
Let $(\sfV, \n{\cdot}_\sfV)$ be a normed linear space over $\R$. Let $I\subset\R$ be an interval and $\gamma:I\to V$ a continuous mapping. Then the pair $(I, \gamma)$ is called a curve in $V$. The family of all curves on a given Interval $I$ is denoted by $\Gamma(I, \sfV)$.
\end{definition}
Note that if $\sfV$ is finite dimensional, i.e., isomorphic to $\R^m$, for some $m\in \N$, then a curve can be interpreted as a vector-valued function $\gamma: I\to\R^m$,
\be\label{vecval}
\gamma(t) = \begin{pmatrix} \gamma_1 (t)\\ \vdots \\ \gamma_m (t)
\end{pmatrix},
\ee
with each $\gamma_i:I\to\R$.

When the interval $I$ is understood, we simply write $\gamma$ instead of $(I,\gamma)$. We need to distinguish between the curve $\gamma$ as a mapping $I\to\sfE$ and the trace of $\gamma$, i.e., the image set $\abs{\gamma} :=\gamma(I)\subset\sfE$.

\begin{definition}
Let $(\sfE, \n{\cdot}_\sfE)$ be a real Banach space and $I\subset\R$ an open interval. Suppose that $\gamma:I \to \sfE$ is a curve and $t_0\in I$. The curve $\gamma$ is called differentiable in $t_0$ if there exists a linear mapping $v:\R\to\sfE$ such that
\[
\lim_{h\to 0} \frac{\n{\gamma(t_0+h) - \gamma(t_0) - v(h)}_\sfE}{h} = 0.
\]
\end{definition}
We note that if a curve $\gamma:I\to\sfE$ is differentiable at a point $t_0$ then the derivative is uniquely determined. 

The linear space
\[
L(\R, \sfE) := \left\{v:\R\to\sfE : \text{$v$ is linear}\right\}
\]
is isomorphic to $\sfE$ and, as every linear mapping $v:\R\to\sfE$ is continuous, $L(\R, \sfE)$ is naturally a Banach space. The isomorphism between the Banach spaces $\sfE$ and $L(\R, \sfE)$ allows us to identify the linear mapping $v:\R\to\sfE$ with the element $v(1)\in \sfE$. The vector $v(1)$ is called the derivative of the curve $\gamma$ at $t_0$ and is denoted by $\gamma'(t_0)$. 
\begin{definition}
Let $(\sfE, \n{\cdot}_\sfE)$ be a real Banach space and $I\subset\R$ an interval. Further, let $\gamma:I\to\sfE$ be a curve. $\gamma$ is called continuously differentiable on $I$ if 
\begin{enumerate}
\item $\gamma$ is differentiable at every point $t\in\Int(I) $, the interior of $I$; 
\item the mapping
\[
I\to L(\R,\sfE), \quad t\mapsto \gamma'(t)
\]
is continuous on $\Int(I)$ and extends continuously to $I$. 
The family of all continuously differentiable curves $\gamma:I\to\sfE$ is denoted by $\Gamma^1(I,\sfE)$.
\end{enumerate}
A continuously differentiable curve $\gamma\in \Gamma^1 (I,\sfE)$ is called \emph{regular} if $\gamma'\neq 0$ on $I$.
\end{definition}

Now consider the case of a Banach algebra $\sfA:=(\sfA, +, \cdot)$, i.e., when $(\sfA, +)$ is a vector space with a norm $\n{\cdot}$ and endowed with a product $\cdot:\sfA\times\sfA$ satisfying
\begin{enumerate}
\item $(\sfA, +, \n{\cdot})$ is a Banach space;
\item $(\sfA, +, \cdot)$ is an associate $\R$-algebra;
\item $\forall\,a_1,a_2\in \sfA: \n{a_1\cdot a_2} \leq \n{a_1}\n{a_2}$.
\end{enumerate}
A Banach algebra $\sfA$ is called \emph{unital} if it has a multiplicative identity.
\begin{example}\label{ex1}
For a unital Banach algebra $\sfA$, we obtain for any $a\in \sfA$ a curve $\gamma : I \to \sfA$ defined by $t\mapsto \exp t \cdot a$, where $\exp$ denotes the exponential function in the Banach algebra $\sfA$. 

As a special case, we consider a finite dimensional Banach algebra $\sfA$ such as $\K^m$ with $\K \in \{\R,\C\}$, for some $m\in \N$, endowed with some norm. Then, $\Hom (\sfA,\sfA)$, the set of self-maps $\sfA\to\sfA$ with multiplication $\cdot$ defined as composition of maps $\circ$, is a Banach algebra when equipped with the operator norm and for each $L\in \Hom(\sfE,\sfE)$, we have curves $\gamma: I\to \Hom(\sfE,\sfE)$, $t\mapsto \exp t\cdot L$ yielding curves $I\to \sfE$, $t\mapsto (\exp t\cdot L)(v)$, for each $v\in \sfE$.
\end{example}

For our later purposes, we require the following \emph{Generalized Mean Value Theorem for Curves}.
\begin{theorem}\label{thm3.5}
Assume that the curve $\gamma:[a,b]\to\sfE$ is differentiable on $(a,b)$ and $\{\gamma'(t) : t\in (a,b)\}$ is bounded. Define
\[
\n{\gamma'}_\sfE := \sup\{\n{\gamma'(t) : t\in (a,b)}\}.
\]
Then,
\be\label{mwt}
\n{\gamma(b) - \gamma(a)}_\sfE \leq \n{\gamma'}_\sfE (b-a).
\ee
\end{theorem}

\section{Fractal Interpolation over Curves}\label{sec5}
We briefly recall some basics from fractal interpolation theory and the theory of fractal functions. 

To this end, let $\sfU$ and $\sfV$ be open subsets of Banach spaces $\sfE$ and $\sfF$, respectively. A mapping $f:\sfU\to\sfV$ is called a \emph{$C^1$-diffeomorphism} if $f$ is a bijection, Fr\'echet differentiable on $\sfU$, and the inverse of $f$, $f^{-1}$, is Fr\'echet differentiable on $\sfV$. In a similar way, one defines higher order diffeomorphisms $\sfU\to\sfV$. 

The collection of all $C^\alpha$-diffeomorphisms from $\sfU\to\sfV$ with $\alpha\in \N_0\cup\{\infty\}$ is denoted by $\Diff^\alpha (\sfU,\sfV)$. In case, $\sfU := \sfE := \sfF$, we simply write $\Diff^\alpha (\sfE)$. 

If $\sfX\subset\sfE$ then $f:\sfX\to\sfF$ is called a $C^\alpha$-diffeomorphism on $\sfX$, in symbols $f\in \Diff^\alpha(\sfX,\sfF)$, if there exists an open $\sfU\subset\sfE$ with $\sfX\subset\sfU$ and a $C^\alpha$-diffeomorphism $g:\sfU\to\sfF$ such that $f = g\lvert_\sfX$. 

\subsection{Fractal Interpolation in General}
Let $\sfX$ be a nonempty bounded subset of a Banach space $\sfE$. Suppose we are given a finite family 
\[
h:= \{h_i\in \Diff^\alpha (\sfX)  : i = 1, \ldots, n\} 
\]
of $C^\alpha$-diffeomorphisms generating a partition $\Pi(h)$ of $\sfX$ in the sense that
\be\label{c1}
\sfX = \coprod_{i=1}^n h_i(\sfX),
\ee
where $\coprod$ denotes the disjoint union of sets. For simplicity, we set $\sfX_i := h_i(\sfX)$.\\

Recall that a mapping $f:\sfE\to\sfF$ is called \emph{affine} if $f - f(0)$ is linear.

\begin{definition}
A partition $\Pi(h)$ of $\sfX$ is called {nonlinear} if the maps generating $\Pi(h)$ are not affine; otherwise linear.
\end{definition}

In case of a nonlinear partition, we may express \eqref{c1} also as follows: The IFS $(\sfX, h)$ has as its attractor $\sfX$ which consists of finitely many disjoint \emph{nonlinear} images of itself. Note that this extends and generalizes known methodologies. For an introduction to fractal interpolation over nonlinear partitions and related topics, we refer the interested reader to \cite{M22}.

In the following, we consider the partition functions $h_i$ as either being linear or nonlinear. A more precise specification will be provided when necessary.
\\

One of the goals of fractal interpolation is the construction of a \emph{global} function 
\be\label{psipart}
\psi:\sfX = \coprod\limits_{i=1}^n \sfX_i\to\sfF
\ee
belonging to some prescribed function space $\sF$ and satisfying $n$ functional equations of the form
\be\label{psieq}
\psi (h_i (x)) = q_i (x) + s_i (x) \psi (x), \quad\text{$x\in\sfX$, $i\in \N_n$},
\ee
where for each $i\in\N_n$, $q_i\in\sF$ and $s_i$ is chosen such that $s_i\cdot\psi\in\sF$. In other words, the global solution is pieced together in a prescribed manner from copies of itself on the subsets $\sfX_i = h_i(\sfX)$ defined by the partition $\Pi(h)$.

We refer to \eqref{psipart} and \eqref{psieq} as the \emph{fractal interpolation problem} and to $\psi$ as a \emph{fractal function}. The latter were first constructed in \cite{B86} in the context of interpolating a given set of real-valued data in $\R^2$.

One possible way to solve the fractal interpolation problem is to consider \eqref{psieq} as the fixed point equation for an associated affine operator acting on an appropriately defined or prescribed function space.
\subsection{Partitioning Regular Curves in Banach Spaces}
Let $\gamma\in \Gamma^1 (I,\sfE)$ be a regular curve where $I := [0,1]\subset\R$. The emphasis $I = [0,1]$ is only a mild restriction as any other half-open interval $[a,b]$ with $a < b$ can be obtained from $I$ via the orientation-preserving diffeomorphism $t\mapsto (b-a) t + a$.

Note that as $\gamma\in \Gamma^1 (I,\sfE)$, $|\gamma| = \gamma(I)$ is compact in $\sfE$ and thus (totally) bounded. In particular, $\n{\gamma'}$ is bounded and Theorem \ref{thm3.5} is applicable.
\begin{remark}
Strictly speaking, a curve is an equivalence class of parametrized curves. The trace of a curve, however, is uniquely defined, as equivalently parametrized curves have the same trace. Moreover, the terms initial and terminal point of a curve are independent of its parametrization and so is its length. In what follows, we will always use the arc length parametrization of a curve
\begin{gather*}
\sigma : I \to [0, \ell(\gamma)],\\
\sigma := \int_0^t \n{\gamma'(\tau)} d\tau,
\end{gather*}
where $\ell(\gamma)$ denotes the (finite) length of $\gamma$.
\end{remark}
Suppose we generate a partition of $|\gamma|\in\sH(\R)$ by means of finitely many (prescribed) set-valued partition functions $H_i: \sH(\R)\to \sH(\R)$ which are $C^\alpha$-diffeomorphisms such that 
\begin{gather}
H_i(|\gamma|) =: |\gamma_i|, \quad i\in \N_n,\\
H_1(\{\gamma(0)\}) = \{\gamma(0)\} =: p_0, \quad H_n(\{\gamma(1)\}) = \{\gamma(1)\} =:p_n,\\
H_i (\{\gamma(0)\}) =: p_i \in |\gamma_i|,\quad H_i (\{\gamma(1)\}) =: p_{i+1} = H_{i+1} (\{\gamma(0)\})\in |\gamma_{i+1}|, \quad i\in \N_{n-1}.\label{eq4.6}
\end{gather}
Note that \eqref{eq4.6} means that the intersection between adjacent traces is one-elemental:
\[
|\gamma_i| \cap |\gamma_{i+1}| = \{p_i\},\quad i\in \N_{n-1}.
\]
As $\gamma\in \Gamma^1 (I,\sfE)$ and regular, the mapping $t\mapsto \sigma = \int\limits_0^t \n{\gamma'(\tau)} d\tau$ is a $C^1$-diffeomor\-phism. Hence, to every $p_i$, $i\in \N_{n-1}$, corresponds a unique $\sigma_i\in [0,\ell(\gamma)]$ and therefore a unique $t_i\in I$. The set $\{t_j : j \in \N_{n-1}\cup\{0\}\}$ defines a partition of $I$ into subintervals $[t_j, t_{j+1}]$ with $|\gamma_{j+1}| = \gamma([t_{j}, t_{j+1}])$. Two adjacent subintervals of the above form have only the point $\{t_{j+1}\}$ in common thus allowing $I$ to be written as the disjoint union
\be\label{eq4.7}
I = \coprod_{i=1}^{n-1} [t_{i-1},t_i) \sqcup [t_{n-1},t_n] =: \coprod_{i=1}^n I_i.
\ee

Conversely, we could partition the interval $I$ via finitely many (linear or nonlinear) partition functions $h_i :I\to I_i$ into subintervals $I_i :=[t_{i-1},t_i]$ generating a partition of $I$ of the form \eqref{eq4.7}. 
As $|\gamma| = \gamma(I)$ and $|\gamma_i| = \gamma(I_i)$ we obtain
\begin{align*}
(\gamma\circ h_i)(I) = \gamma(h_i (I)) = \gamma(I_i) = |\gamma_i| = H_i (|\gamma|) = H_i (\gamma(I)) = (H_i\circ\gamma)(I),
\end{align*}
and, thus, 
\be\label{eq4.8}
\gamma\circ h_i = H_i\circ\gamma, \quad i\in\N_n.
\ee

Hence, for a regular $\gamma\in \Gamma^1 (I,\sfE)$,
\be\label{gammapart}
|\gamma| = \bigcup_{i=1}^n |\gamma\circ h_i| =: \bigcup_{i=1}^n |\gamma_i|,
\ee
with $|\gamma_i| \cap |\gamma_{i+1}|$ being one-elemental. We can obtain a disjoint partition of $|\gamma|$ by using
\[
|\gamma| = \coprod_{i=1}^{n-1} |\gamma_i|\setminus\{p_i\} \sqcup |\gamma_n|
\]
Clearly, in each of the two cases, $|\gamma_i|$ is compact in $\sfE$ and thus bounded. 
\subsection{The Fractal Interpolation Problem for Curves}
In this subsection, we pose and solve the fractal interpolation problem for regular curves in Banach spaces. More precisely, given the interval $I = [0,1)$, a regular $\gamma\in \Gamma^1 (I,\sfE)$, and a partition $\Pi(h)$ of $|\gamma|$ of the form \eqref{gammapart}, we like to construct a unique function 
\be\label{psigamma}
\psi: |\gamma| = \coprod\limits_{i=1}^n |\gamma_i|\to\sfF
\ee
belonging to some prescribed function space $\sF$ and satisfying $n$ functional equations of the form
\be\label{psieqgamma}
\psi (\gamma_i) = q_i\circ\gamma + (s_i \circ\gamma)\cdot (\psi\circ\gamma), \quad i\in \N_n,
\ee
where for each $i\in\N_n$, $q_i\in\sF$ and $s_i$ is chosen such that $s_i\cdot\psi\in\sF$, i.e., the global solution of \eqref{psieqgamma} is pieced together in a prescribed manner from copies of itself on the subcurves $|\gamma_i| $ defined by the partition $\Pi(h)$. The unique solution of \eqref{psieqgamma} will be termed a \emph{fractal function on $\gamma$}.

Let $(\sfX, d)$ be a metric space and let $\Lip (\sfX, \sfF)$ consist of all bounded Lipschitz functions $\sfX\to\sfF$ endowed with the norm
\be
\n{f} := \n{f}_\infty + L(f).
\ee
Here, $\n{\cdot}_\infty$ denotes the supremum norm on continuous functions $\sfX\to\sfF$. Under this norm and the point-wise product, $\Lip(\sfX,\sfF)$ becomes a unital commutative Banach algebra and is an example of a \emph{Lipschitz algebra}. (Cf., for instance, \cite{J70,Sh63,W18}.)

In what follows, we regard $|\gamma|$ as a (compact) subset of the complete metric space $(\sfE, d)$ where $d$ is the metric induced by the norm $\n{\cdot}_\sfE$. The natural function space $\sF$ to consider here is then the Lipschitz algebra $\Lip (|\gamma|,\sfA)$ with a Banach algebra $\sfA$.

On the Lipschitz algebra $\Lip (|\gamma|,\sfA)$, we define an affine operator $T$, called a Read-Bajractarevi\'c (RB) operator, as follows. Given any $f\in\Lip (|\gamma|,\sfA)$, set
\be\label{eq4.12}
Tf (\gamma_i) := q_i (\gamma) + s_i(\gamma)\cdot f(\gamma), \quad i\in \N_n,
\ee
where $q_i, s_i\in \Lip (|\gamma|,\sfA)$. Equivalently, we may express \eqref{eq4.12} in the form
\be
T (f\circ\gamma) = \sum_{i=1}^n (q_i\circ\gamma)\circ h_i^{-1} \,\one_{I_i} + \sum_{i=1}^n (s_i\circ\gamma)\circ h_i^{-1}\cdot (f\circ\gamma)\circ h_i^{-1} \,\one_{I_i}
\ee
Moreover, we require that the following compatibility (join-up) conditions hold:
\begin{gather}
T f(p_0) = f(p_0), \quad T f(p_n) = f(p_n),\label{eq4.14}\\
T f(p_i-) = T f(p_i+), \;\; i \in \N_{n-1}.\label{eq4.15}
\end{gather}
Eqns. \eqref{eq4.14} and \eqref{eq4.15} can also be expressed in the form
\begin{gather}
f(p_0) = \frac{q_1(p_0)}{1-s_1(p_0)}, \quad f(p_n) = \frac{q_n(p_n)}{1-s_n(p_n)},\\
q_i(p_n) + \frac{s_i(p_n)\cdot q_n(p_n)}{1-s_n(p_n)} = q_{i+1}(p_0) +  \frac{s_{i+1}(p_0)\cdot q_1(p_0)}{1-s_1(p_0)}, \;\; i \in \N_{n-1}.
\end{gather}

Clearly, $Tf\in \Lip (|\gamma|,\sfA)$, i.e., $Tf$ is bounded and continuous, and $\n{Tf} < \infty$.
\begin{theorem}\label{thm4.3}
The RB operator defined in \eqref{eq4.12} and satisfying the compatibility conditions \eqref{eq4.14} and \eqref{eq4.15} is contractive on the Lipschitz algebra $\Lip(|\gamma|, \sfA)$ provided that
\be
L(T) := \max\left\{\max_{i\in \N_n} \n{s_i}_\infty, \max_{i\in\N_n} L(s_i) \cdot \n{\gamma'}_{\sfE}\right\} < 1.
\ee
\end{theorem}
\begin{proof}
It suffices to consider the linear part $T-T(0)$ of $T$. For this purpose, we set $h := f - g$, for $f,g\in \Lip (|\gamma|,\sfA)$.

Now, let $t\in I$. Then, there exists an $i\in \N_n$ such that $t\in I_i$ and thus
\begin{align*}
\n{T h(\gamma(t))}_\sfA \leq \n{s_i(\gamma\circ h_i^{-1}(t))}_\sfA\cdot \n{h(\gamma\circ h_i^{-1}(t))}_\sfA.
\end{align*}
Taking the supremum over $t$ yields
\be\label{estimate1}
\n{T h}_\infty \leq \max_{i\in \N_i}\n{s_i}_\infty\cdot \n{h}_\infty.
\ee
Furthermore, for $s,t\in I$,
\begin{align*}
\n{T h(\gamma(t)) - T h(\gamma(s))}_\sfA & \leq \n{s_i(\gamma(t)) - s_i(\gamma(s))}_\sfA \, \n{h(\gamma(t)) - h(\gamma(s))}_\sfA\\
& \leq L(s_i) \n{\gamma(t) - \gamma(s)}_\sfE \, L(h)  \n{\gamma(t) - \gamma(s)}_\sfE\\
& \leq \max_{i\in \N_n} L(s_i) \,\n{\gamma'}_\sfE\, L(h)\,\n{\gamma(t) - \gamma(s)}_\sfE,
\end{align*}
where we used Theorem \eqref{thm3.5} to obtain the last inequality. Hence, 
\be\label{estimate2}
L(T h) \leq \max_{i\in \N_n} L(s_i) \,\n{\gamma'}_\sfE\,L(h).
\ee
The inequalities \eqref{estimate1} and \eqref{estimate2} now yield the claim.
\end{proof}
Theorem \eqref{thm4.3} yields the following immediate corollary.
\begin{corollary}
The RB operator defined in \eqref{eq4.12} and satisfying the compatibility conditions \eqref{eq4.14} and \eqref{eq4.15} has a unique fixed point $\psi\in \Lip(|\gamma|,\sfA)$. This fixed point obeys the self-referential equation
\[
\psi (\gamma_i) := q_i (\gamma) + s_i(\gamma)\cdot \psi(\gamma), \quad i\in \N_n,
\]
and thus is the unique solution to the fractal interpolation problem \eqref{psipart} and \eqref{psieq} in the present setting.

Moreover, $\psi : |\gamma|\to\sfA$ can be obtained by the iterative procedure
\begin{gather*}
\psi_0 \in \Lip(|\gamma|,\sfA)\text{ chosen arbitrarily.};\\
\psi_k := T\psi_{k-1}, \;k\in \N,
\end{gather*}
with
\[
\psi = \lim_{k\to\infty} \psi_k,
\]
where the limit is taken in the norm topology of $\Lip(|\gamma|, \sfA)$.

Furthermore, one has the error estimate
\[
\n{\psi - \psi_k} \leq \frac{L(T)^k}{1-L(T)}\,\n{T\psi_0 - \psi_0},\quad k\in \N.
\]
\end{corollary}
\begin{proof}
The statements are straight-forward consequences of the Banach Fixed Point Theorem and its constructive proof.
\end{proof}
\begin{definition}
The unique fixed point $\psi$ of the RB operator \eqref{eq4.12} is called an $\sfA$-valued fractal function over the curve $\gamma$.
\end{definition}
The next examples show the flexibility and versatility of the new concept introduced above.
\begin{example}
In the case $\sfE:= \R^m$, we can regard a curve $\gamma:I\to\R^m$ as the vector-valued function \eqref{vecval}. Considering $\R^m$ also as a metric space $(\R^m,d)$ with the metric $d$ given by $d(x,y) = \n{x-y} = \sqrt{\sum\limits_{j=1}^m (x_j-y_j)^2}$ and taking for $\sfA$ the commutative unital Banach algebra $\R^k$ with norm 
\[
\n{x}_\sfA := \max_{l\in \N_k} |x_l|,\quad x = (x_1,\ldots, x_k)^\top,
\]
where ${}^\top$ denotes the transpose, and multiplication defined by
\[
x\cdot y := (x_1 y_1, \ldots, x_k y_k)^\top,
\]
the elements of the Lipschitz algebra $\Lip(|\gamma|, \R^k)$ are then vector-valued functions of $m$ variables of the form
\begin{gather*}
f: |\gamma|\subset\R^m\to \R^k,\\
f(\gamma(t)) = \begin{pmatrix} f_1(\gamma_1(t),\ldots, \gamma_m(t))\\ \vdots \\f_k(\gamma_1(t),\ldots, \gamma_m(t))
\end{pmatrix},
\end{gather*}
where $f_l:\R^m\to\R$, $l\in \N_k$.
\end{example}
\begin{example}
Here we take again $\sfE := \R^m$ but for $\sfA$ we choose $\R^{k\times k}$, the collection of all $k\times k$-matrices with entries in $\R$ or $\C$. As a norm on $\R^{k\times k}$ we select a submultiplicative matrix norm such as, for instance, a Schatten $p$-norm,
\[
\n{A}_p := \left(\sum_{l=1}^k \sigma_l^p (A)\right)^{1/p}, \quad A\in \R^{k\times k}, \;\; p\in [1,\infty],
\]
where the $\sigma_l$ denote the singular values of $A$. Multiplication is the usual multiplication of matrices. Then $\R^{k\times k}$ becomes a unital noncommutative Banach algebra.

In this setting, the elements of the Lipschitz algebra $\Lip(|\gamma|, \R^{k\times k})$ are matrix-valued functions of the form
\begin{gather*}
f: |\gamma|\subset\R^m\to \R^{k\times k},\\
f(\gamma(t)) = \begin{pmatrix} f_{11}(\gamma(t)) & \cdots & f_{1k}(\gamma(t))\\ \vdots & \ddots & \vdots\\
f_{k1}(\gamma(t)) & \cdots & f_{kk}(\gamma(t))
\end{pmatrix},
\end{gather*}
with $f_{ij}:\R^m\to\R$, $i,j\in \N_k$.
\end{example}
\section{Relation to Iterated Function Systems}
In this section, the relation between the graph $G(\psi)$ of the fixed point $\psi$ of the operator $T$ and the attractor of the associated contractive IFS is described. 

To this end, note that there are two IFSs whose attractors are the interval $I$ and the trace $|\gamma|$ of a curve $\gamma:I\to\sfE$, respectively. The former is given by $(I, h)$ with $h:=\{h_i : i\in \N_n\}$ and the latter by $(|\gamma|, H)$ where $H:=\{H_i : i\in \N_n\}$. Both IFSs have the same code space, namely, $\N_n^\infty$, and as \eqref{eq4.8} holds, we see that the fractal transformation $\fF : I \to |\gamma|$ is indeed a homeomorphism. In other words, for the purposes of this section, it does not matter which IFS we consider for the first variable in $G(\psi)$. In order to keep the presentation simple, we choose the IFS $(|\gamma|, H)$ and leave the other case to the interested reader.

Next, consider the product Banach space $\sfE\times\sfA$ endowed with the product norm $\n{\cdot}_\sfE + \n{\cdot}_\sfA$. For $|\gamma|\in\sH(\sfE)$ and $\sfY\in\sH(\sfA)$, let mappings $w_i:|\gamma|\times\sfY\to|\gamma|\times\sfA$ be defined by
\be\label{wn}
w_i (x, y) := (H_i (x), q_i (x) + s_i(x)\cdot y), \quad i\in\N_n,
\ee
and required to map into $|\gamma|\times\sfY$. (By Proposition 27 in \cite{M16},  this is always possible.) Here, $q_i$ and $s_i$ are defined as in Section \ref{sec5}.

Note that the $H_i$ are defined on the compact subset $|\gamma|$ of $\sfE$ and that $H_i(|\gamma|)\cap H_j(|\gamma|)\neq\emptyset$ but one-elemental for $i,j\in \N_n$ with $i\neq j$. However, we insist that $H$ is \emph{non-overlapping} with respect to the IFS $(|\gamma|, H)$ in the sense of \cite[Definition 2.4]{BBHV16} and that at the contact points $\{H_i(|\gamma|)\cap H_j(|\gamma|)\}$ the compatibility conditions \eqref{eq4.14} and \eqref{eq4.15} hold. This will guarantee the existence of a solution $\psi$ (cf. also \cite{SB17}).

To facilitate notation, introduce the mappings 
\begin{gather*}
v_i : \sfE\times\sfY\to\sfA,\\
v_i := q_i + s_i\cdot y,\quad i\in \N_n.
\end{gather*}

We remark that for a fixed $x\in |\gamma|$, arbitrary $y_1,y_2\in \sfY$, and $i\in \N_n$,
\begin{align*}
\n{v_i(x,y_1) - v_i(x,y_2)}_\sfA &= \n{s_i(x) y_1 - s_i(x) y_2}_\sfA\\
& \leq \n{s_i}_\sfA \n{y_1-y_2}_\sfA \leq L(T)  \n{y_1-y_2}_\sfA, 
\end{align*}
i.e., $v_i$ is uniformly contractive in the second variable.

Moreover, for a fixed $y\in \sfY$, arbitrary $x_1, x_2\in |\gamma|$, and $i\in \N_n$,
$v_i$ is also uniformly Lipschitz continuous in the first variable:
\begin{align*}
\n{v_i(x_1, y) - v_i(x_2, y)}_\sfA &\leq \n{q_i(x_1)-q_i(x_2)}_\sfA + \n{y}_\sfA \n{s_i(x_1)-s_i(x_2)}_\sfA\\
& \leq \left(L(q_i) + \n{y}_\sfA L(s_i) \right) \n{x_1-x_2}_\sfA.
\end{align*}
As $y\in \sfY\in \sH(\sfA)$, there exists a nonnegative constant $c_\sfY$ such that $\n{y}_\sfA \leq c_\sfY$. Hence,
\[
\n{v_i(x_1, y) - v_i(x_2, y)}_\sfA \leq\left(L(q_i) + c_\sfY L(s_i) \right) \n{x_1-x_2} _\sfA =: \lambda \n{x_1-x_2}_\sfA,
\]
proving the claim.

Let $\oL (H) := \max\{L(H_i) : i\in\N_n\}$ and let $\vartheta := \frac{1-\oL (H)}{2\lambda}$. Then the mapping $\n{\cdot}_\vartheta : \sfE\times\sfA\to \R_0^+$ given by
\[
\n{\cdot}_\vartheta := \n{\cdot}_\sfE + \vartheta\,\n{\cdot}_\sfF
\]
is a norm on $\sfE\times\sfA$ compatible with the product topology on $\sfE\times \sfA$.

The next theorem can also be found in \cite{BHM14} in a more general setting. Here, we adopt it for our purposes.

\begin{theorem}\label{thm3.1}
Set $\cW := \{w_1, \ldots, w_n\}$. Then, $\cF := (|\gamma|\times\sfY, \cW)$ is a contractive IFS with respect to the norm $\n{\cdot}_\vartheta$ and the graph $G(\psi)$ of the solution to the fractal interpolation problem
\be\label{eq3.8}
\psi(\gamma_i) = q_i(\gamma)+s_i(\gamma),\quad i\in\N_n,
\ee
%
is the unique attractor of the IFS $\cF$. 

Furthermore, if $T$ is the RB operator 
\begin{gather}\label{eq3.9a}
T: \Lip(|\gamma|,\sfA)\to\Lip(|\gamma|,\sfA),\nonumber\\
f\mapsto  q_i\circ\gamma\circ h_i^{-1} + (s_i\circ\gamma\circ h_i^{-1})\cdot (f\circ\gamma\circ h_i^{-1}),\;\;\text{on $\sfX_i$},\; i\in\N_n,
\end{gather}
associated with the fractal interpolation problem \eqref{eq3.8} then
\be\label{GW}
G(T f) = \cF (G(f)),\quad\forall\,f\in \Lip(|\gamma|,\sfA),
\ee
where $\cF$ denotes the set-valued operator \eqref{fixedpoint}.
\end{theorem}
\begin{proof}
See \cite[Theorem 3.7]{BHM14} adapted to the current setting.
\end{proof}

The above statements are paramount to the commutativity of the diagram below.
\be\label{diagram}
\begin{CD}
|\gamma|\times \sfA @>\cF>> |\gamma|\times \sfA\\
@AAGA                  @AAGA\\
\Lip(|\gamma|,\sfA) @>T>>  \Lip(|\gamma|,\sfA)
\end{CD}
\ee
\nl
where $G$ is the mapping $\Lip(|\gamma|,\sfA)\ni g\mapsto G(g) = \{(x, g(x)) : x\in |\gamma|\}\in \sfE\times \sfA$.

On the other hand, if one assumes that $\cF = (|\gamma|\times\sfA, w_1, w_2, \ldots, w_n)$ is an IFS whose mappings $w_i$ are of the form \eqref{wn} where the functions $H_i$ are linear or nonlinear partition functions $H_i:|\gamma |\to|\gamma_i|$ with non-overlapping attractor $|\gamma|$ and contact points $\{H_i(|\gamma|)\cap H_j(|\gamma|)\}$. Further assume that the mappings $v_i$ are  uniformly Lipschitz continuous in the first variable and uniformly contractive in the second variable. 

Then we can associate with the IFS $\cF$ an RB operator $T$ of the form \eqref{eq3.9a} and thus a fractal interpolation problem $H_i(|\gamma|)\cap H_j(|\gamma|)\neq\emptyset$ with appropriate compatibility conditions. The attractor $A$ of $\cF$ is then the graph $G(\psi)$ of the solution $\psi$ of \eqref{eq3.8}, respectively, the fixed point of $T$.
\section{Future Research Directions}
The above approach may be generalized along different directions. 
\begin{itemize}
\item Instead of considering curves in Banach spaces, one may look in general at Banach or even Fr\'echet submanifolds.
\item The Lipschitz algebras $\Lip (X, \sfF)$ could be replaced by the space of bounded real-valued Lipschitz functions on $X$ which vanish at a given base point $x_0\in X$. The corresponding Lipschitz algebras are then denoted by $\Lip_0(X,\sfF)$.
\end{itemize}
We leave it to the interested reader to pursue these research directions.
\bibliographystyle{amsalpha}

\end{document}